\documentclass[a4paper,12pt]{article}
\usepackage{mathrsfs}

\usepackage{amsfonts}
\usepackage{amscd,color}
\usepackage{amsmath,amsfonts,amssymb,amscd}
\usepackage{indentfirst,graphicx,epsfig}
\usepackage{graphicx,psfrag}
\input{epsf}
\usepackage{graphicx}
\usepackage{epstopdf}
\usepackage{caption}

\newtheorem{thm}{Theorem}[section]

\newtheorem{lem}[thm]{Lemma}

\newtheorem{cor}[thm]{Corollary}

\newenvironment {proof} {\noindent{\em Proof.}}{\hspace*{\fill}$\Box$\par\vspace{4mm}}

\title{The $(k,\ell)$-proper index of graphs\footnote{Supported by NSFC No.11371205, 11531011, ``973" program No.2013CB834204 and PCSIRT.}}

\author{Hong Chang$^1$, Xueliang Li$^1$, Colton Magnant$^2$, Zhongmei Qin$^1$\\
{\small $^1$Center for Combinatorics and LPMC}\\
{\small Nankai University, Tianjin 300071, P.R. China}\\
{\small Email: changh@mail.nankai.edu.cn, lxl@nankai.edu.cn, qinzhongmei90@163.com}\\
{\small $^2$Department of Mathematical Sciences}\\
{\small Georgia Southern University, Statesboro, GA 30460-8093, USA}\\
{\small Email: cmagnant@georgiasouthern.edu}
}
\date{}
\begin{document}
\maketitle
\begin{abstract}
A tree $T$ in an edge-colored graph is called a {\it proper tree} if no two adjacent edges of $T$ receive the same color. Let $G$ be a connected graph of order $n$ and $k$ be an integer with $2\leq k \leq n$. For $S\subseteq V(G)$ and $|S| \ge 2$, an $S$-tree is a tree containing the vertices of $S$ in $G$.  Suppose $\{T_1,T_2,\ldots,T_\ell\}$ is a set of $S$-trees, they are called \emph{internally disjoint} if $E(T_i)\cap E(T_j)=\emptyset$ and $V(T_i)\cap V(T_j)=S$ for $1\leq i\neq j\leq \ell$. For a set $S$ of $k$ vertices of $G$, the maximum number of internally disjoint $S$-trees in $G$ is denoted by $\kappa(S)$. The $\kappa$-connectivity $\kappa_k(G)$ of $G$ is defined by $\kappa_k(G)=\min\{\kappa(S)\mid S$ is a $k$-subset of $V(G)\}$. For a connected graph $G$ of order $n$ and for two integers $k$ and $\ell$ with $2\le k\le n$ and $1\leq \ell \leq \kappa_k(G)$, the \emph{$(k,\ell)$-proper index $px_{k,\ell}(G)$} of $G$ is the minimum number of colors that are needed in an edge-coloring of $G$ such that for every $k$-subset $S$ of $V(G)$, there exist $\ell$ internally disjoint proper $S$-trees connecting them. In this paper, we show that for every pair of positive integers $k$ and $\ell$ with $k \ge 3$, there exists a positive integer $N_1=N_1(k,\ell)$ such that $px_{k,\ell}(K_n) = 2$ for every integer $n \ge N_1$, and also there exists a positive integer $N_2=N_2(k,\ell)$ such that $px_{k,\ell}(K_{m,n}) = 2$ for every integer $n \ge N_2$ and $m=O(n^r) (r \ge 1)$. In addition, we show that for every $p \ge c\sqrt[k]{\frac{\log_a n}{n}}$ ($c \ge 5$),  $px_{k,\ell}(G_{n,p})\le 2$ holds almost surely, where $G_{n,p}$ is the Erd\"{o}s-R\'{e}nyi random graph model. \\[2mm]
\end{abstract}

\section{Introduction}

All graphs in this paper are undirected, finite and simple. We follow \cite{BM} for graph theoretical notation and terminology not described here. Let $G$ be a nontrivial connected graph with an associated {\it edge-coloring} $c : E(G)\rightarrow \{1, 2, \ldots, r\}$, $r \in \mathbb{N}$, where adjacent edges may have the same color. If adjacent edges of $G$ are assigned different colors by $c$, then $c$ is a {\it proper coloring}. The minimum number of colors needed in a proper coloring of $G$ is referred to as the \emph{chromatic index} of $G$ and denoted by $\chi'(G)$. A path of $G$ is said to be a {\it rainbow path} if no two edges on the path receive the same color. The graph $G$ is called {\it rainbow connected} if for every pair of distinct vertices there is a rainbow path of $G$ connecting them. An edge-coloring of a connected graph is a {\it rainbow connecting coloring} if it makes the graph rainbow connected. This concept of rainbow connection of graphs was introduced by Chartrand et al.~\cite{CJMZ} in 2008. The \emph{rainbow connection number} $rc(G)$ of a connected graph $G$ is the smallest number of colors that are needed in order to make $G$ rainbow connected. The readers who are interested in this topic can see \cite{LSS,LS} for a survey.

In \cite{COZ}, Chartrand et al. generalized the concept of rainbow connection to rainbow index. At first, we recall the concept of generalized connectivity. Let $G$ be a connected graph of order $n$ and $k$ be an integer with $2\leq k \leq n$. For $S\subseteq V(G)$ and $|S| \ge 2$, an $S$-tree is a tree containing the vertices of $S$ in $G$.  Let $\{T_1,T_2,\ldots,T_\ell\}$ be a set of $S$-trees, they are called \emph{internally disjoint} if $E(T_i)\cap E(T_j)=\emptyset$ and $V(T_i)\cap V(T_j)=S$ for every pair of distinct integers $i,j$ with $1\leq i, j\leq \ell$. For a set $S$ of $k$ vertices of $G$, the maximum number of internally disjoint $S$-trees in $G$ is denoted by $\kappa(S)$. The $\kappa$-connectivity $\kappa_k(G)$ of $G$ is defined by $\kappa_k(G)=\min\{\kappa(S)\mid S$ is a $k$-subset of $V(G)\}$.  We refer to a book \cite{LM} for more details about the generalized connectivity.

A tree $T$ in an edge-colored graph is a {\it rainbow tree} if no two edges of $T$ have the same color. Let $G$ be a connected graph of order $n$ and let $k,\ell$ be two positive integers with $2\le k\le n$ and $1\leq \ell \leq \kappa_k(G)$. The \emph{$(k,\ell)$-rainbow index} of a connected graph $G$, denoted by $rx_{k,\ell}(G)$, is the minimum number of colors that are needed in an edge-coloring of $G$ such that for every $k$-subset $S$ of $V(G)$, there exist $\ell$ internally disjoint rainbow $S$-trees connecting them. Recently, a lot of relevant results have been published in \cite{CLS,CLS2,CLS3}. In particular, for $\ell=1$, we write $rx_k(G)$ for $rx_{k,1}(G)$ and call it the \emph{$k$-rainbow index} of $G$ ( see \cite{CLZ, QXY,CLYZ}).

Motivated by rainbow coloring and proper coloring in graphs, Andrews et al.~\cite{ALLZ} and Borozan et al.~\cite{BFGMMMT} introduced the concept of proper-path coloring. Let $G$ be a nontrivial connected graph with an edge-coloring. A path in $G$ is called a \emph{proper path} if no two adjacent edges of the path are colored with the same color. An edge-colored graph $G$ is {\it proper connected} if any two vertices of $G$ are connected by a proper path. For a connected graph $G$, the {\it proper connection number} of $G$, denoted by $pc(G)$, is defined as the smallest number of colors that are needed in order to make $G$ proper connected. The {\it $k$-proper connection number} of a connected graph $G$, denoted by $pc_k(G)$, is the minimum number of colors that are needed in an edge-coloring of $G$ such that for every two distinct vertices of $G$ are connected by $k$ internally pairwise vertex-disjoint proper paths. For more details, we refer to \cite{GLQ,LWY} and a dynamic survey \cite{LC}.

Recently, Chen et al. \cite{CLL} introduced the concept of $k$-proper index of a connected graph $G$. A tree $T$ in an edge-colored graph is a {\it proper tree} if no two adjacent edges of $T$ receive the same color. Let $G$ be a connected graph of order $n$ and $k$ be a fixed integer with $2\le k\le n$. An edge-coloring of $G$ is called a \emph{$k$-proper coloring} if for every $k$-subset $S$ of $V(G)$, there exists a proper $S$-tree in $G$. For a connected graph $G$, the \emph{$k$-proper index} of $G$, denoted by $px_k(G)$, is defined as the minimum number of colors that are needed in a $k$-proper coloring of $G$. In \cite{CLQ}, we gave some upper bounds for the 3-proper index of graphs.

A natural idea is to introduce the concept of the $(k,\ell)$-proper index. Let $G$ be a nontrivial connected graph of order $n$ and size $m$. Given two integers $k,\ell$ with $2\le k\le n$ and $1\leq \ell \leq \kappa_k(G)$, the \emph{$(k,\ell)$-proper index} of a connected graph $G$, denoted by $px_{k,\ell}(G)$, is the minimum number of colors that are needed in an edge-coloring of $G$ such that for every $k$-subset $S$ of $V(G)$, there exist $\ell$ internally disjoint proper $S$-trees connecting them. From the definition, it follows that $$1 \le px_{k,\ell}(G) \le \min\{rx_{k,\ell(G)}, \chi'(G) \}\le m.$$ Clearly, $px_{2,\ell}(G)=pc_\ell(G)$, $px_{k,1}(G)=px_k(G)$.

Let us give an overview of the rest of this paper. In Section 2, we study the $(k,\ell)$-proper index of complete graphs using two distinct methods. We show that there exists a positive integer $N_1=N_1(k,\ell)$, such that $px_{k,\ell}(K_n)=2$ for every integer $n\geq N_1$. In Section 3, we turn to investigate the $(k,\ell)$-proper index of complete bipartite graphs by probabilistic method \cite{AS}. Similarly, we prove that there exists a positive integer $N_2=N_2(k,\ell)$, such that $px_{k,\ell}(K_{n,n})=2$ for every integer $n\geq N_2$. Furthermore, we can extend the result about $K_{n,n}$ to more general complete bipartite graph $K_{m,n}$, where $m=O(n^r)$, $r\in\mathbb{R}$ and $r\geq 1$. In section 4, we show that for every $p \ge c\sqrt[k]{\frac{\log_a n}{n}}$ ($c \ge 5$), $px_{k,\ell}(G_{n,p})\le 2$ holds almost surely, where $G_{n,p}$ is the Erd\"{o}s-R\'{e}nyi random graph model \cite{ER}.

\section{Complete graphs}

In this section, we will investigate the $(k,\ell)$-proper index of complete graphs. Firstly, we state a useful result about the $k$-connectivity of $K_n$ and present some preliminary results.

\begin{thm}\label{thm0}\cite{COZ}
For every two integers $n$ and $k$ with $2 \leq k \leq n$, $\kappa_k(K_n)=n-\lceil\frac{k}{2}\rceil$.
\end{thm}

\begin{thm}\cite{CLL}
Let $G=K_n$ and $k$ be an integer with $3 \le k \le n$. Then $px_{k,1}(G)=px_k(G)=2$.
\end{thm}

\begin{thm}\label{thm1}\cite{BFGMMMT}
Let $G=K_n$, $n\geq4$ and $\ell\ge 2$. If $n\geq2\ell$, then $px_{2,\ell}(G)=pc_\ell(G)=2$.
\end{thm}

Based on the previous results, we prove the following.
\begin{thm}
For every integer $n\ge 4$, $px_{3,2}(K_n)=2$.
\end{thm}
\begin{proof} {\bf Case 1.} $n=2p$ for $p \ge 2$.

Take a Hamiltonian cycle $C=v_1,v_2, \ldots, v_{2p}$ of $K_n$ and denote the $v_sv_t$-path in clockwise direction contained in $C$ by $C[v_s,v_t]$. Next, we will provide an edge-coloring of $K_n$ with 2 colors such that for any three vertices of $K_n$, there are two internally disjoint  proper trees connecting them. We alternately color the edges of $C$ with colors 1 and 2 starting with color 2, and color the rest of the edges with color 1. Let $S$ be any 3-subset of $V(K_n)$, without loss of generality, we assume that $S=\{v_i, v_j , v_h\}$ with $1 \le i < j < h \le 2p$. It is easy to see that the Hamiltonian cycle $C$ is partitioned into three segments $C[v_i,v_j]$, $C[v_j,v_h]$ and $C[v_h, v_i]$. If $2p+i-h =1$, then we get that $i=1$ and $h=2p$. Note that the edges $v_1v_2$ and $v_{2p-1}v_{2p}$ are colored with color 2 and the edges $v_{2p}v_1$ and $v_{2p-1}v_1$ are colored with color 1. Thus, there are two internally disjoint proper $S$-trees $P_1=v_{2p}v_1C[v_1,v_j]$ and $P_2=v_jv_{2p}v_{2p-1}v_1$ (if $v_j=v_{2p-1}$, then $P_2=v_{2p}v_{2p-1}v_1$). Now we suppose that $2p+i-h \ge 2$. If the edge incident to $v_i$ with color 2 is $v_iv_{i+1}$, then the two internally disjoint proper $S$-trees are $v_hv_iC[v_i,v_j]$ and $C[v_j,v_h]C[v_h,v_i]$. Otherwise, the edge incident to $v_i$ with color 2 is $v_{i}v_{i-1}$. Thus, the two internally disjoint proper $S$-trees are $v_hv_iv_{i-1}v_j$ and $C[v_i,v_j]C[v_j,v_h]$.

{\bf Case 2.} $n=2p+1$ for $p \ge 2$.

Let $H$ be a complete subgraph of $K_n$ with vertex set $\{v_1, v_2, \ldots, v_{2p}\}$. Color the edges of $H$ as above, and color the edge $v_{2p+1}v_i$ ($1\leq i\leq 2p$) with color 1 for $i$ odd and with color 2 for $i$ even. It is easy to see that for any three vertices of $V(H)$, there are two internally disjoint proper trees connecting them. Now we assume that $S=\{v_i, v_j ,v_{2p+1}\}$ ($1 \le i < j \le 2p$). Since there are two internally disjoint proper paths connecting $v_i$ and $v_j$ in the Hamiltonian cycle $C=v_1,v_2, \ldots, v_{2p}$, it follows that there are two internally disjoint proper $S$-trees in $K_n$.
\end{proof}

\begin{thm}
For every integer $n\ge 4$, $px_{n-1,2}(K_n)=2$.
\end{thm}
\begin{proof} It is well known that if $n$ is even, then $K_n$ can be factored into $\frac{n}{2}$ Hamiltonian paths $\{P_1,P_2,\ldots,P_{\frac{n}{2}}\}$; if $n$ is odd, then $K_n$ can be factored into $\frac{n-1}{2}$ Hamiltonian cycles $\{C_1,C_2,\ldots,C_{\frac{n-1}{2}}\}$. Let $S$ be any $(n-1)$-subset of $V(K_n)$. Without loss of generality, assume that $\{v\}=V(K_n)\setminus S$. If $n$ is even, then for each Hamiltonian path $P_i$, we alternately color the edges of $P_i$ with color 1 and 2. Notice that each vertex will be an end-vertex of a Hamiltonian path. Let $P_i$ be the Hamiltonian path which contains $v$ as one of its end-vertices. Thus, there are two internally disjoint proper $S$-trees $P_i-v$ and $P_j$ $(i\neq j)$. If $n$ is odd, then for each Hamiltonian cycle $C_i$, we alternately color the edges of $C_i$ with color 1 and 2 such that the edges incident to $v_1$ in each cycle are colored the same. If $v=v_1$, then there are two internally disjoint proper $S$-trees $C_1-v_1$ and $C_2-v_1$. Now we suppose $v \ne v_1$. Let $C_i$ be the Hamiltonian cycle in which $v$ is adjacent to $v_1$, and let $v'$ be one of the neighbours of $v_1$ in another Hamiltonian cycle $C_j (i \ne j)$. Thus, there are two internally disjoint proper $S$-trees $C_i- v$ and $C_j - v_1v'$. Hence, $px_{n-1,2}(K_n)=2$.
\end{proof}


\begin{thm}\label{thm2}
For every two integers $n$ and $\ell$ with $n\geq2\ell$, $px_{n,\ell}(K_n)=2$.
\end{thm}

\begin{proof} It is obvious that $px_{n,\ell}(K_n)\geq2$. To show that the converse inequality, we will provide an edge-coloring of $K_n$ with 2 colors such that there are at least $\ell$ internally disjoint spanning proper trees in $K_n$. It is well known that if $n$ is even, then $K_n$ can be factored into $\frac{n}{2}$ Hamiltonian paths; if $n$ is odd, then $K_n$ can be factored into $\frac{n-1}{2}$ Hamiltonian cycles. In conclusion, $K_n$ contains $\lfloor\frac{n}{2}\rfloor$ pairwise edge-disjoint Hamiltonian paths. Thus, for each Hamiltonian path, we alternately color the edges using color 1 and 2 starting with color 1. If there still remains uncolored edges, then color them with color 1. It is easy to see that there exist $\lfloor\frac{n}{2}\rfloor \ge \ell$ internally disjoint spanning proper trees. Hence, $px_{n,\ell}(K_n)\leq2$. This completes the proof.
\end{proof}

\noindent\textbf{Remark:} Theorem \ref{thm2} is best possible in the sense of the order of $K_n$. It follows from Theorem \ref{thm0} that there are at most $\lfloor\frac{n}{2}\rfloor$ internally disjoint trees connecting $n$ vertices of $K_n$. Thus, $\ell\leq \lfloor\frac{n}{2}\rfloor$.

These results naturally lead to the following question: for general integers $k, \ell$ with $3\leq k\leq n-1$, whether there exists a positive integer $N_1=N_1(k,\ell)$ such that $px_{k,\ell}(K_n)=2$ for every integer $n\geq N_1$? Now, we use two distinct methods to answer this question. Although the result of Theorem~\ref{Thm:Kn-2-Explicit} is stronger than the result of Theorem~\ref{thm3}, we include both as a demonstration of the variety of possible approaches to this kind of question.

In order to prove Theorem~\ref{Thm:Kn-2-Explicit}, we need the following result of Sauer as presented in \cite{BB}.

\begin{thm} \cite{S} \label{Thm:Sauer}
Given $\delta \geq 3$ and $g \geq 3$, for any
$$
m \geq \frac{(\delta - 1)^{g - 1} - 1}{\delta - 2},
$$
there exists a $\delta$-regular graph $G$ of order $2m$ with girth at least $g$.
\end{thm}

By removing a vertex from the graph provided by Theorem~\ref{Thm:Sauer}, we obtain an almost regular graph of odd order, still with girth at least $g$, but the degree of some vertices is $\delta - 1$. Thus, we replace $\delta$ with $\delta + 1$ in Theorem~\ref{Thm:Sauer} to obtain the following easy corollary.

\begin{cor}\label{Cor:Sauer}
Given $\delta \geq 3$ and $g \geq 3$, for any
$$
n \geq 2\frac{\delta^{g - 1} - 1}{\delta - 1} - 1,
$$
there exists a graph $G$ of order $n$ with $\delta(G) \geq \delta$ and girth at least $g$.
\end{cor}

\begin{thm}\label{Thm:Kn-2-Explicit}
Let $k \geq 3$ and $\ell \geq 1$. For all $n$ with
$$
n \geq 2 \frac{ (\ell(k - 1) + k)^{4} - 1 }{(\ell + 1)(k - 1)} - 1,
$$
we have $px_{k, \ell}(K_{n}) = 2$.
\end{thm}

\begin{proof}
First note that a proper tree using only two colors must be a path. This means that the goal of this result is to produce an edge-coloring with two colors of a complete graph in which any set of $k$ vertices is contained in $\ell$ internally disjoint proper paths.

By Corollary~\ref{Cor:Sauer} with $g = 5$ and $\delta = \ell(k - 1) + k$, we have that for any
$$
n \geq 2 \frac{ (\ell(k - 1) + k)^{4} - 1 }{(\ell + 1)(k - 1)} - 1,
$$
there exists a graph $H$ with $n$ vertices, girth at least $5$, and minimum degree at least $\ell(k - 1) + k$. Color the edges of $H$ red and color the complement of $H$ blue to complete a coloring of $G = K_{n}$.

Let $S$ be any set of $k$ vertices in this graph, say with $S = \{v_{1}, v_{2}, \dots, v_{k}\}$. We say that a path $P$ is $S$-alternating if each odd vertex of $P$ is in $S$ while each even vertex of $P$ is in $G \setminus S$. Suppose that we have constructed $t \geq 0$ proper paths that are $S$-alternating, each contain all of $S$, and are vertex-disjoint aside from the vertices of $S$. If $t = \ell$, this is the desired system of paths so suppose $t < \ell$ and choose the constructed set of paths so that $t$ is as large as possible.

Further suppose we have constructed an additional $S$-alternating proper path $P^{i}$ from $v_{1}$ to $v_{i}$ using red edges of the form $v_{j}w_{j}$ for each $j < i$ where $w_{j} \in G \setminus S$ and blue edges of the form $w_{j}v_{j + 1}$. If $i = k$, this contradicts the choice of $t$ so suppose $i < k$ and further choose this constructed path so that $i$ is as large as possible. Note that at most $\ell(k - 1) - 1$ vertices of $G \setminus S$ have been used in these existing paths.

Incident to $v_{i}$, there are at least $\delta - (\ell (k - 1) - 1) - (k - 1) \geq 2$ red edges to vertices of $G \setminus S$ that have not already been used in paths. Let $x$ and $y$ be the opposite ends of these edges. At most one of $x$ and $y$, say $x$, may have a red edge to $v_{i + 1}$ since the red graph was constructed to have girth $5$. This means that the proper path $P^{i}$ can be extended to $P^{i + 1}$ by including the red edge $v_{i}y$ and the blue edge $yv_{i + 1}$. This contradiction completes the proof.
\end{proof}

Next, we answer the above question by using probabilistic method.

\begin{thm}\label{thm3}
Let $k \geq 3$ and $\ell \geq 1$. For all $n$ with
$$
n \geq 2k(k + \ell) \ln \left( \frac{1}{1 - (1/2)^{2k - 3}} \right),
$$
we have $px_{k, \ell}(K_{n}) = 2$.
\end{thm}

\begin{proof} Obviously, $px_{k,\ell}(K_n)\geq2$. For the converse, we color the edges of $K_n$ with two colors uniformly at random. For a $k$-subset $S$ of $V(K_n)$, let $A_S$ be the event that there exist at least $\ell$ internally disjoint proper $S$-trees. Note that a proper tree using only two colors must be a path. It is sufficient to show that $Pr[\ \underset{S} \bigcap A_S\ ]>0$.

Let $S$ be any $k$-subset of $V(K_n)$, without loss of generality, we assume $S=\{v_1,v_2,\ldots,v_k\}$. For any $(k-1)$-subset $T$ of $V(K_n)\setminus S$, suppose $T=\{u_1,u_2,\ldots,u_{k-1}\}$, define $P_T={v_1u_1v_2u_2\cdots v_{k-1}u_{k-1}v_k}$ as a path of length $2k-2$ from $v_1$ to $v_k$, which implies $P_T$ is an $S$-tree. Note that for $T,T'\subseteq V(K_n)\setminus S$ and $T\cap T'=\emptyset$, $P_T$ and $P_{T'}$ are two internally disjoint $S$-trees.  Let $\mathcal{P}=\{P_T \mid T\subseteq V(K_n)\setminus S\}$. Take $\mathcal{P}'$ to be a subset of $\mathcal{P}$ which consists of $\lfloor\frac{n-k}{k-1}\rfloor$ internally disjoint $S$-trees in $\mathcal{P}$. Set $p= Pr[\ P_T \in\mathcal{P}'$ is a proper $S$-tree\ ]$=\frac{2}{2^{2k-2}}=\frac{1}{2^{2k-3}}$. Let $A_S'$ be the event that there exist at most $\ell-1$ internally disjoint proper $S$-trees in $\mathcal{P}'$. Assume that $\lfloor\frac{n-k}{k-1}\rfloor > \ell-1$ (that is, $n\ge k+(k-1)\ell$), we have
\begin{align*}
Pr[\ \overline{A_S}\ ]&\leq Pr[\ A_S'\ ]\\
&\leq {\lfloor\frac{n-k}{k-1}\rfloor\choose \lfloor\frac{n-k}{k-1}\rfloor-(\ell-1)}(1-p)^{\lfloor\frac{n-k}{k-1}\rfloor-(\ell-1)}\\
&= {\lfloor\frac{n-k}{k-1}\rfloor\choose \ell-1}(1-p)^{\lfloor\frac{n-k}{k-1}\rfloor-(\ell-1)}.
\end{align*}

Then over all possible choices of $S$ with $|S| = k$, we get
\begin{align*}
Pr[\ \underset{S} \bigcap A_S\ ]&=1-Pr[ \ \bigcup\overline{A_S} \ ]\\
&\geq1-\underset{S} \sum Pr[\overline{A_S}]\\
&>1-{n\choose k}{\lfloor\frac{n-k}{k-1}\rfloor\choose \ell-1}(1-p)^{\lfloor\frac{n-k}{k-1}\rfloor-(\ell-1)}\\
&>1-n^k\left\lfloor\frac{n-k}{k-1}\right\rfloor^{\ell-1}(1-p)^{\lfloor\frac{n-k}{k-1}\rfloor-\ell+1}\\
& >0
\end{align*}
for
$$
n \geq 2k(k + \ell) \ln \left( \frac{1}{1 - (1/2)^{2k - 3}} \right).
$$

\end{proof}

\section{Complete bipartite graphs}

In this section, we will turn to study the $(k,\ell)$-proper index of complete bipartite graphs.

\begin{thm}\label{thm4}
Let $k$ and $\ell$ be two positive integers with $k\geq3$. Then there exists a positive integer $N_2$ such that $px_{k,\ell}(K_{n,n})=2$ for every integer $n\geq N_2$.
\end{thm}

\begin{proof} Clearly, $px_{k,\ell}(K_{n,n})\geq2$. We only need to show that $px_{k,\ell}(K_{n,n})\leq 2$. Firstly, color the edges of $K_{n,n}$ with two colors uniformly at random. For a $k$-subset $S$ of $V(K_{n,n})$, let $B_S$ denote the event that there exist at least $\ell$ internally disjoint proper $S$-trees. Note that a proper tree using only two colors must be a path. It is sufficient to show that $Pr[\ \underset{S} \bigcap B_S\ ]>0$.

Assume that $K_{n,n}=G[X,Y]$, where $X=\{x_1, x_2, \ldots, x_n\}$ and $Y=\{y_1, y_2, \ldots, y_n\}$. We distinguish the following two cases.

\noindent{\bf Case 1.}  Fix the vertices in $S$ in the same class of $K_{n,n}$.

Without loss of generality, we suppose that $S=\{x_1,x_2,\ldots,x_k\}\subseteq X$. Let $T$ be any $(k-1)$-subset of $Y$, assume that $T=\{y_1,y_2,\ldots, y_{k-1}\}$, define $P_T=x_1y_1\ldots x_{k-1}y_{k-1}x_k$ as a path of length $2k-2$ from $x_1$ to $x_k$, which follows that $P_T$ is an $S$- tree. Note that for $T,T'\subseteq Y$ and $T\cap T'=\emptyset$, $P_T$ and $P_{T'}$ are two internally disjoint $S$-trees.  Let $\mathcal{P}_1=\{P_T \mid T\subseteq V(K_n)\setminus S\}$. Take $\mathcal{P}_1'$ to be a subset of $\mathcal{P}_1$ which consists of $\lfloor\frac{n}{k-1}\rfloor$ internally disjoint $S$-trees in $\mathcal{P}_1$. Set $p_1= Pr[\ P_{T} \in\mathcal{P}_1'$ is a proper $S$-tree\ ]$=\frac{2}{2^{2k-2}}=\frac{1}{2^{2k-3}}$. Let $B_S'$ be the event that there exist at most $\ell-1$ internally disjoint proper $S$-trees in $\mathcal{P}_1'$.  Assume that $\lfloor\frac{n}{k-1}\rfloor > \ell-1$ (that is, $n\ge(k-1)\ell$), we have
\begin{align*}
Pr[\ \overline{B_S}\ ]&\leq Pr[\ B_S'\ ]\\
&\leq {\lfloor\frac{n}{k-1}\rfloor\choose \lfloor\frac{n}{k-1}\rfloor-(\ell-1)}(1-p_1)^{\lfloor\frac{n}{k-1}\rfloor-(\ell-1)}\\
&= {\lfloor\frac{n}{k-1}\rfloor\choose \ell-1}(1-p_1)^{\lfloor\frac{n}{k-1}\rfloor-(\ell-1)}\\
&<(\frac{n}{2})^{\ell-1}(1-p_1)^{\lfloor\frac{n}{k-1}\rfloor-(\ell-1)}.
\end{align*}

\noindent{\bf Case 2.} Fix the vertices in $S$ in different classes of $K_{n,n}$.

Suppose that $S\cap X=\{s_1,s_2,\ldots, s_r\}$ and $S\cap Y=\{s_{r+1},s_{r+2},\ldots, s_k\}$, where $r$ is an positive integer with $1\leq r\leq k-1$. For any $(k-r)$-subset $T_X$ of $X\setminus S$ with $T_X=\{x_1,x_2,\ldots, x_{k-r}\}$ and any $r$-subset $T_Y$ of $Y\setminus S$ with $T_Y=\{y_1,y_2,\ldots, y_{r}\}$, let $T=T_X\cup T_Y$, and define $P_T=s_1y_1\ldots s_{r}y_{r}x_1s_{r+1}x_2\ldots s_{k-1}x_{k-r}s_k$ as a path of length $2k-1$ from $s_1$ to $s_k$. Obviously, $P_T$ is an $S$-tree. Let $\mathcal{P}_2=\{P_T \mid T=T_X\cup T_Y ,T_X\subseteq X\setminus S,T_Y\subseteq Y\setminus S\}$. Then $\mathcal{P}_2$ has $t=\min\{\lfloor\frac{n-r}{k-r}\rfloor, \lfloor\frac{n-(k-r)}{r}\rfloor\}$ internally disjoint $S$-trees. Take $\mathcal{P}_2'$
 to be a subset of $\mathcal{P}_2$ which consists of $t$ internally disjoint $S$-trees in $\mathcal{P}_2$. Set $p_2= Pr[\ P_T \in\mathcal{P}_2'$ is a proper $S$-tree\ ]$=\frac{2}{2^{2k-1}}=\frac{1}{2^{2k-2}}$. Let $B_S''$ be the event that there exist at most $\ell-1$ internally disjoint proper $S$-trees in $\mathcal{P}_2'$. Note that $\lfloor\frac{n-1}{k-1}\rfloor \le t < \frac{n}{2}$. Here, we assume that $\lfloor\frac{n-1}{k-1}\rfloor > \ell-1$ (that is, $n\ge (k-1)\ell+1$), we obtain
\begin{align*}
Pr[\ \overline{B_S}\ ]&\leq Pr[\ B_S''\ ]\\
&\leq {t \choose t-(\ell-1)}(1-p_2)^{t-(\ell-1)}\\
&= {t\choose \ell-1}(1-p_2)^{t-(\ell-1)}\\
&<(\frac{n}{2})^{\ell-1}(1-p_2)^{\lfloor\frac{n-1}{k-1}\rfloor-(\ell-1)}.
\end{align*}

Since $p_2<p_1$ and $\lfloor\frac{n-1}{k-1}\rfloor \le \lfloor\frac{n}{k-1}\rfloor$, we get $Pr[\ \overline{B_S}\ ]<(\frac{n}{2})^{\ell-1}(1-p_2)^{\lfloor\frac{n-1}{k-1}\rfloor-(\ell-1)}$ for every $k$-subset S of $V(K_{n,n})$. It yields that
\begin{align*}
Pr[\ \underset{S} \bigcap B_S\ ]&=1-Pr[\ \bigcup\overline{B_S}\ ]\\
&\geq1-\underset{S} \sum Pr[\overline{B_S}]\\
&>1-{2n\choose k}(\frac{n}{2})^{\ell-1}(1-p_2)^{\lfloor\frac{n-1}{k-1}\rfloor-(\ell-1)}\\
&>1-2^{k-\ell+1}n^{k+\ell-1}(1-p_2)^{\lfloor\frac{n-1}{k-1}\rfloor-(\ell-1)}.
\end{align*}

We shall be guided by the above inequality in search for the value of $N_2$, we find that the inequality $2^{k-\ell+1}n^{k+\ell-1}(1-p_2)^{\lfloor\frac{n-1}{k-1}\rfloor-(\ell-1)}\leq1$ will lead to $Pr[\ \underset{S} \bigcap B_S\ ]>0$.
With similar arguments in Theorem \ref{thm3}, we can obtain that there exists a positive integer $N_2$ for every integer $n\geq N_2$. 
\end{proof}

With arguments similar to Theorem \ref{thm4}, we can extend the above result to more general complete bipartite graph $K_{m,n}$, where $m=O(n^r)$, $r\in\mathbb{R}$ and $r\geq 1$.

\begin{thm}\label{thm5}
Let $m$ and $n$ be two positive integers with $m=O(n^r)$, $r\in\mathbb{R}$ and $r\geq 1$. For every pair of integer of $k,\ell$ with $k\geq3$, there exists a positive integer $N_3=N_3(k,\ell)$ such that $px_{k,\ell}(K_{m,n})=2$ for every integer $n\geq N_3$.
\end{thm}

\section{Random graphs}

At the beginning of this section, we introduce some basic definitions about random graphs. The most frequently occurring probability model of random graphs is the Erd\"{o}s-R\'{e}nyi random graphs model \cite{ER}. The model $G_{n,p}$ consists of all graphs on $n$ vertices in which the edges are chosen independently and randomly with probability $p$. We say that an event $\mathcal{A}$ happens almost surely if $Pr[\mathcal{A}]\rightarrow1$ as $n\rightarrow\infty$.

We now focus on the $(k,\ell)$-proper index of the random graph $G_{n,p}$. In what follows, we first show two lemmas which are useful in the main result of this section.

\begin{lem}\label{lem2}
 For any $c\geq 5$, if $p\geq c\sqrt[k]{\frac{\log_a n}{n}}$, then almost surely any $k$ vertices in $G_{n,p}$ have at least $2k^2\log_a n$ common neighbours, where $a=1+\frac{1}{2^{2k-3}-1}$.
\end{lem}

\begin{proof} For a $k$-subset $S$ of $V(G_{n,p})$, let $C_S$ be the event that all the vertices in $S$ have at least $2k^2\log_a n$ common neighbours. It is sufficient to prove that for $p= c\sqrt[k]{\frac{\log_a n}{n}}$, $Pr[ \ \underset{S} \bigcap C_S\ ]\rightarrow1,n\rightarrow\infty$. Let $C_1$ be the number of common neighbours of all the vertices in $S$. Then $C_1\sim B\left(n-k,\left(c\sqrt[k]{\frac{\log_a n}{n}}\right)^k\right)$, and $E(C_1)=\frac{n-k}{n}c^k\log_a n$. In order to apply the Chernoff bound \cite{JLR} as follows, setting $n>\frac{kc^k}{c^k-2k^2}$.

By the Chernoff Bound, we obtain
\begin{align*}
Pr[\ \overline{C_S}\ ]&= Pr[\ C_1< 2k^2\log_a n\ ]\\
&=Pr[\ C_1<E(C_1)\left(1-\frac{E(C_1)-2k^2\log_a n}{E(C_1)}\right)\ ]\\
&=Pr[\ C_1<\frac{n-k}{n}c^k\log_a n\left(1-\frac{(n-k)c^k-2k^2n}{(n-k)c^k}\right)\ ]\\
&\leq e^{-\frac{n-k}{2n}c^k\log_a n\left(\frac{(n-k)c^k-2k^2n}{(n-k)c^k}\right)^2}\\
&<n^{-\frac{c^k(n-k)}{2n}\left(\frac{(n-k)c^k-2k^2n}{(n-k)c^k}\right)^2}.
\end{align*}
Since $1<a<e$, we have $\log_a n>\ln n$, this leads to the last inequality.

As an immediate consequence, we get
\begin{align*}
Pr[\ \underset{S} \bigcap C_S\ ]&=1-Pr[\ \underset{S}\bigcup\overline{C_S}\ ]\\
&\geq1-\underset{S} \sum Pr[\ \overline{C_S}\ ]\\
&>1-{n\choose k}n^{-\frac{c^k(n-k)}{2n}(\frac{(n-k)c^k-2k^2n}{(n-k)c^k})^2}\\
&>1-n^{k-\frac{c^k(n-k)}{2n}(\frac{(n-k)c^k-2k^2n}{(n-k)c^k})^2}.
\end{align*}
Note that for any $c\geq5$, $k-\frac{c^k(n-k)}{2n}(\frac{(n-k)c^k-2k^2n}{(n-k)c^k})^2<0$ holds for sufficiently large $n$. Thus, $\underset{n\rightarrow\infty}\lim Pr[\ \underset{S} \bigcap C_S\ ]=\underset{n\rightarrow\infty}\lim 1-n^{k-\frac{c^k(n-k)}{2n}(\frac{(n-k)c^k-2k^2n}{(n-k)c^k})^2}=1$.
\end{proof}

\begin{lem}\label{lem3}
Let $a=1+\frac{1}{2^{2k-3}-1}$. If any $k$ vertices in $G_{n,p}$ have at least $2k^2\log_a n$ common neighbours, then $px_{k,\ell}(G_{n,p})\le 2$ holds almost surely.
\end{lem}

\begin{proof} Firstly, we color the edges of $G_{n,p}$ with two colors uniformly at random. For a $k$-subset $S$ of $V(G_{n,p})$, let $D_S$ be the event that there exist at least $\ell$ internally disjoint proper $S$-trees. Note that a proper tree using only two colors must be a path. If $Pr[\ \underset{S} \bigcap D_S\ ]>0$, then a suitable coloring of $G_{n,p}$ with 2 colors exists, which follows that $px_{k,\ell}(G_{n,p})\leq 2$.

We assume that $S=\{v_1,v_2,\ldots,v_k\}\subseteq V(K_n)$, let $N(S)$ be the set of common neighbours of all vertices in $S$. Let $T$ be any $(k-1)$-subset of $N(S)$, without loss of generality, suppose $T=\{u_1,u_2,\ldots,u_{k-1}\}$, define $P_T={v_1u_1v_2u_2\cdots v_{k-1}u_{k-1}v_k}$ as a path of length $2k-2$ from $v_1$ to $v_k$. Obviously, $P_T$ is an $S$-tree. Let $\mathcal{P}^\star=\{P_T \mid T \subseteq N(S)\}$, then $\mathcal{P}^\star$ has at least $\lfloor\frac{2k^2\log_a n}{k-1}\rfloor\geq 2k\log_a n$ internally disjoint $S$-trees (we may and will assume that $2k\log_a n$ is an integer). Take $\mathcal{P}^\star_1$
to be a set of $2k\log_a n$ internally disjoint $S$-trees of $\mathcal{P}^\star$. It is easy to check that $q=$ Pr[\ $P_T \in\mathcal{P}^\star_1$ is a proper $S$-tree\ ]$=\frac{2}{2^{2k-2}}=\frac{1}{2^{2k-3}}$. So $1-q=a^{-1}$. Let $D_1$ be the number of proper $S$-trees in $\mathcal{P}^\star_1$. Then we get
\begin{align*}
Pr[\ \overline{D_S}\ ]&\leq Pr[\ D_1\leq \ell-1\ ]\\
&\leq{2k\log_a n\choose 2k\log_a n-(\ell-1)}(1-q)^{2k\log_a n-(\ell-1)}\\
&={2k\log_a n\choose \ell-1}(1-q)^{2k\log_a n-(\ell-1)}\\
&< (2k\log_a n)^{\ell-1}a^{-(2k\log_a n-(\ell-1))}\\
&=\frac{(2ak\log_a n)^{\ell-1}}{n^{2k}}.
\end{align*}

Consequently
\begin{align*}
Pr[\ \underset{S} \bigcap D_S\ ]&=1-Pr[\ \bigcup\overline{D_S}\ ]\\
&\geq1-\underset{S} \sum Pr[\ \overline{D_S}\ ]\\
&\geq1-{n\choose k}\frac{(2ak\log_a n)^{\ell-1}}{n^{2k}}\\
&>1-\frac{(2ak\log_a n)^{\ell-1}}{n^{k}}.
\end{align*}
It is easy to verify that $\underset{n\rightarrow\infty}\lim 1-\frac{(2ak\log_a n)^{\ell-1}}{n^{k}}=1$, which implies that $\underset{n\rightarrow\infty}\lim Pr[\ \underset{S} \bigcap D_S\ ]=1$, this is to say that $px_{k,\ell}(G_{n,p})\leq 2$ holds almost surely. This completes the proof.
\end{proof}

Combining with Lemmas \ref{lem2} and \ref{lem3}, we get the following conclusion.

\begin{thm}\label{thm6}
Let $a=1+\frac{1}{2^{2k-3}-1}$ and $c \ge 5$. For every $p \ge c\sqrt[k]{\frac{\log_a n}{n}}$,  $px_{k,\ell}(G_{n,p})\le 2$ holds almost surely.
\end{thm}

\end{document}